\documentclass{amsart}

\usepackage{amsmath}
\usepackage{latexsym}
\usepackage{amssymb}
\usepackage{setspace}
\usepackage{xcolor}
\usepackage{ifthen}
\usepackage{dsfont}
\usepackage{mathtools}
\usepackage{tikz}
\usepackage{hyperref}
\usetikzlibrary{positioning}
\usepackage[T1]{fontenc}

\newtheorem{Thm}{Theorem}[section]
\newtheorem{Lem}[Thm]{Lemma}

\newtheorem{Cor}[Thm]{Corollary}

\newtheorem{thmA}{Theorem}

\newtheorem{propA}[thmA]{Proposition}

\theoremstyle{definition}
\newtheorem{Def}[Thm]{Definition}

\newtheorem{Qu}[Thm]{Question}

\theoremstyle{remark}
\newtheorem{Rem}[Thm]{Remark}

\DeclareMathOperator{\CAT}{CAT}

\numberwithin{equation}{section}

\newcommand{\id}{\mathds{1}}

\newcommand{\RR}{\mathbb{R}}

\newcommand{\ZZ}{\mathbb{Z}}
\newcommand{\NN}{\mathbb{N}}

\makeatletter
\@namedef{subjclassname@2020}{\textup{2020} Mathematics Subject Classification}
\makeatother

\usepackage{mathtools}

\newcommand{\eps}{\varepsilon}

\DeclareMathOperator{\cl}{\mathbf{cl}}
\DeclareMathOperator{\sell}{s\ell}
\DeclareMathOperator{\scl}{\mathbf{scl}}
\DeclareMathOperator{\rl}{\mathbf{rl}}
\DeclareMathOperator{\srl}{\mathbf{srl}}
\DeclareMathOperator{\tl}{\mathbf{tl}}
\DeclareMathOperator{\stl}{\mathbf{stl}}

\DeclareMathOperator{\Pc}{\mathrm{Pc}}
\DeclareMathOperator{\core}{\mathrm{core}}

\newboolean{ComsOn}
\setboolean{ComsOn}{true}
\ifthenelse{\boolean{ComsOn}}{\newcommand{\com}[2][blue]{\textcolor{#1}{#2}}}{\newcommand{\com}[2][blue]{}}

\definecolor{amethyst}{rgb}{0.6, 0.4, 0.8}
\definecolor{kellygreen}{rgb}{0.3, 0.73, 0.09}

\begin{document}

\title[srl in Coxeter groups]{Stable reflection length in Coxeter groups}

\author{Francesco Fournier-Facio}
\address{Department of Pure Mathematics and Mathematica Statistics, University of Cambridge, United Kingdom}
\email{ff373@cam.ac.uk}

\author{Marco Lotz}
\address{Heidelberg University,
Faculty of Mathematics and Computer Science, 
Im Neuenheimer Feld 205,
69120 Heidelberg}
\email{mlotz@mathi.uni-heidelberg.de}

\author{Timoth{\'e}e Marquis}
\address{Universit{\'e} Catholique de Louvain, IRMP, 1348 Louvain-la-Neuve, Belgium}
\email{timothee.marquis@uclouvain.be}

\thanks{FFF is supported by the Herchel Smith Postdoctoral Fellowship Fund. ML is supported in part by DFG Grant – 314838170, GRK 2297 MathCoRe. TM is a F.R.S.-FNRS Research associate, and is supported in part by the FWO and the F.R.S.-FNRS under the EOS programme (project ID 40007542).}


\date{\today.}


\begin{abstract}
We introduce stable reflection length in Coxeter groups, as a way to study the asymptotic behaviour of reflection length. This creates connections to other well-studied stable length functions in groups, namely stable commutator length and stable torsion length. As an application, we give a complete characterisation of elements whose reflection length is unbounded on powers.
\end{abstract}

\maketitle

\section*{Introduction}

A natural geometric function associated to a Coxeter group $W$ is its \emph{reflection length}, denoted $\rl$: the reflection length of an element $w \in W$ is the minimal number of hyperplane reflections needed to reflect the fundamental chamber $C$ onto the chamber $wC$ in the Coxeter complex of $W$.
This statistic is well understood for finite Coxeter groups \cite{Bessis2003} and was initially investigated for affine Coxeter groups, where it was shown to be bounded \cite{Mccam2011}. For Coxeter groups that are a direct product of finite and affine reflection groups, there are formulas for $\rl$ \cite{Carter1972}, \cite{Lewis2018}. Together with the fact that $\rl$ is additive under products, its study reduces to the case of Coxeter groups of irreducible indefinite type (Lemma \ref{lem reduction}).

The interest shifted to asymptotic behaviours with the result of Duszenko that in a Coxeter group of indefinite type, $\rl$ is unbounded \cite{Duszenko2011}. Since then, some work went into understanding to what extent this unboundedness could be witnessed by cyclic subgroups, in particular those generated by Coxeter elements. This was initially achieved by Drake and Peters for universal Coxeter groups \cite{Drake2021} and then by the second author for Coxeter groups with sufficiently large labels \cite{Lotz2024}.

Our main result gives a full characterisation for arbitrary elements in arbitrary Coxeter groups. This involves the notion of \emph{straight part} from \cite{marquis:long}. We refer the reader to Subsection \ref{ss straight} for the definition. For now let us just recall that an element $w$ is \emph{straight} if $\ell_S(w) = n \cdot \ell_S(w)$ for all $n \geq 1$, and point out that a straight element is equal to its own straight part.

\begin{thmA}
\label{intro thm main}
    Let $W$ be a Coxeter group and $w \in W$. Let $\Pc(w) = P_1 \times \cdots \times P_r$ be the decomposition of the parabolic closure of $w$ into irreducible components. Write $w = w_1 \cdots w_r$ with $w_i \in P_i$. Then exactly one of the following holds:
    \begin{enumerate}
        \item For each $i \in \{1, \ldots, r\}$ such that $P_i$ is of indefinite type, the straight part of $w_i$ is a product of two involutions.
        \item The reflection length $\rl(w^n)$ grows linearly, in particular it is unbounded.
    \end{enumerate}
\end{thmA}

In earlier works, the focus was on finding sufficient conditions on the Coxeter graph of a Coxeter group that would ensure the unboundedness of $\rl(w^n)$ for every Coxeter element $w$ \cite{Drake2021, Lotz2024}. We obtain a full combinatorial characterisation of when this holds.

\begin{thmA}[Theorem \ref{thm cliques}]
\label{intro thm combinatorial}
    Let $(W, S)$ be a Coxeter system of irreducible indefinite type with Coxeter graph $\Gamma$.
    \begin{enumerate}
        \item There exists a Coxeter element $w \in W$ such that $\{\rl(w^n)\}_{n \geq 1}$ is bounded, if and only if $\Gamma$ is bipartite.
        \item $\{ \rl(w^n) \}_{n \geq 1}$ is bounded for every Coxeter element $w \in W$, if and only if $\Gamma$ is a tree.
    \end{enumerate}
\end{thmA}

We only state this in the case of irreducible indefinite type, but the general case reduces to this one (Corollary \ref{cor cliques general}). Note that this gives a wealth of examples of Coxeter groups of irreducible indefinite type with the property that $\{\rl(w^n)\}_{n \geq 1}$ is bounded for some Coxeter elements, and unbounded for others.

\medskip

A novelty in our approach lies in shifting the focus from reflection length to its stabilisation. We call this the \emph{stable reflection length}, denoted $\srl$. Then $\srl(w) > 0$ if and only if $\rl(w^n)$ grows linearly, and in particular is unbounded. This is analogous to other stable length functions that have a rich theory, most importantly \emph{stable commutator length} $\scl$ \cite{Calegari2009} and \emph{stable torsion length} $\stl$ \cite{Avery2023}. In fact, these quantities are intimately connected. For $\srl$ and $\stl$, this takes the form of a bi-Lipschitz equivalence (Lemma \ref{lem stl srl bilip}). For $\scl$, this is less direct, but there is still a strong connection in the generic case, which is a main step towards Theorem \ref{intro thm main}.

\begin{propA}[Corollary \ref{cor main rank one}]
\label{intro thm srl scl}
    Let $W$ be a Coxeter group of irreducible indefinite type. Then for all $w \in W$ with $\Pc(w) = W$, the following are equivalent.
    \begin{itemize}
        \item $\scl(w) = 0$;
        \item $\srl(w) = 0$;
        \item The straight part of $w$ is a product of two involutions.
    \end{itemize}
\end{propA}

Therefore one can interpret $\srl$ as a tool that creates a bridge between the rich literature on $\scl$ and the geometry of reflection length. We hope that this will be useful beyond the problem at hand. Below, we propose two motivating questions for future research (Subsection \ref{ss questions}).

\medskip

\paragraph{\bf{Outline.}} In Section \ref{sec. prelim} we go over some preliminaries on Coxeter groups and length functions. In Section \ref{sec. scl} we reduce the positivity of $\scl$ to an algebraic property: \emph{chirality}. In Section \ref{sec. main} we characterise this in terms of products of involutions, proving Theorem \ref{intro thm main} and Proposition \ref{intro thm srl scl}. Finally, in Section \ref{sec. combinatorial} we focus on Coxeter elements, proving Theorem \ref{intro thm combinatorial}.

\medskip

\paragraph{\bf{Acknowledgements.}} 
The first two authors thank Raphael Appenzeller, Lvzhou Chen, Petra Schwer and Henry Wilton for useful discussions.

\section{Preliminaries}
\label{sec. prelim}

\subsection{Coxeter groups}

The basic theory of Coxeter groups is treated in detail in e.g. \cite{Davis2012}.

\begin{Def}
Let $\Gamma_0 = (S, E)$ be a finite graph with vertex set $S= \{s_1, \dots, s_n\}$, edge set $E= \{\{u,v\}\subseteq S \mid u\neq v\}$ and an edge-labelling function $m:E\to \mathbb{N}_{\geq 2}\cup \{\infty\}$. We abbreviate $m({s_i,s_j})$ with $m_{ij}$. The corresponding \emph{Coxeter group} $W$ is given by the presentation
\begin{align*}
    W = \langle S \mid \; & \; s_i^2 = \id \;\; \forall i \in \{1, \ldots, n\}, \\
    &(s_is_j)^{m_{ij}}= \id \;\forall\, i\neq j\in \{1,\dots, n\} \text{ with } m_{ij} < \infty \rangle.
\end{align*}
The pair $(W,S)$ is called a \emph{Coxeter system}. The graph $\Gamma$ obtained from $\Gamma_0$ by omitting edges with label $2$ is called the \emph{Coxeter graph} of $(W,S)$.

We denote by $\ell_S(w)$ the minimal length of a word in the generators $S$ representing an element $w\in W$.

Given a subset $I \subset S$, the induced subgraph of $\Gamma_0$ with vertex set $I$ defines a Coxeter system $(P, I)$, and $P$ is isomorphic to the subgroup $W_I$ of $W$ generated by $I$. A subgroup of $W$ is called \emph{parabolic} if it is conjugate to $W_I$ for some $I \subset S$. For $w \in W$ there exists a smallest parabolic subgroup containing $w$, called the \emph{parabolic closure} of $w$ and denoted $\Pc(w)$.

The group $W$ is called \emph{irreducible} if $\Gamma$ is connected. This is equivalent to $W$ not decomposing as a direct product of two Coxeter groups defined on proper subgraphs. The irreducible Coxeter systems split into three families: \emph{finite} type (when the group is finite), \emph{affine} type (when the group is infinite and virtually abelian), and \emph{indefinite} type (in all other cases). If $W$ is finite (not necessarily irreducible) it is called \emph{spherical}.
\end{Def}

\subsection{Straight elements}
\label{ss straight}

Let $W$ be of irreducible indefinite type. An element $w \in W$ is called \emph{straight} if $\ell_S(w^n) = n \cdot \ell_S(w)$. The geometry of straight elements is especially well-behaved, so it is useful to extract straight elements out of arbitrary elements.

Suppose that $\Pc(w) = W$. By \cite[Theorem~9.6]{marquis:long}, there is a largest spherical parabolic subgroup $P_w^{\max}$ of $W$ normalised by $w$. As in \cite[Definition~9.21]{marquis:long}, we associate to $w$ its \emph{core} $w_c=\core(w)$, so that $w$ has a unique decomposition of the form $w=aw_c^n$ with $n \geq 1$ and $a \in P_w^{\max}$. This decomposition is called the \emph{core splitting} of $w$. The element $w_{\infty}:=w_c^n$ is then called the \emph{straight part} of $w$ --- this terminology is motivated by the fact that if $w$ is straight then $w=w_{\infty}$, see \cite[Remark~9.25]{marquis:long}. More generally, $w$ is straight if and only if $w=w_{\infty}$ and $w$ is cyclically reduced, see \cite[Corollary~8.11]{marquis:long}. See also \cite[Lemma~8.9]{marquis:long} for a more geometric definition of the straight part.

We collect here a few properties of cores and core splittings from \cite{marquis:long}.

\begin{Lem}[\cite{marquis:long}]
\label{lem straight part}
Let $w\in W$ with $\Pc(w)=W$. Then:
\begin{enumerate}
    \item $\Pc(w_c)=\Pc(w^m)=\Pc(aw)=W$ and $P_w^{\max}=P_{w_c}^{\max}=P_{w^m}^{\max}=P_{aw}^{\max}$ for all $m\neq 0$ and $a\in P_w^{\max}$.
    \item There are some $n, N \geq 1$ such that $w^N=w_c^{nN}$.
    \item $\core(w^m)=\core(w)$ for all $m \geq 1$, and $\core(w^{-1})=\core(w)^{-1}$.
    \item Write $P_w^{\max}=vW_Iv^{-1}$ for some spherical subset $I\subseteq S$ and some $v\in W$ of minimal length in $vW_I$. Then $\core(v^{-1} wv)=v^{-1}\core(w)v$.
    \item If $v$ commutes with $w$, then $v=aw_c^n$ for some $a\in P_w^{\max}$ and $n\in\ZZ$.
\end{enumerate}
\end{Lem}

\begin{proof}
    (1) follows from \cite[Lemma~9.16(3)]{marquis:long} and (2) from the core splitting $w=aw_c^n$ of $w$. The first part of (3) follows from \cite[Lemma~9.23]{marquis:long}, and its second part from the definition of the core (\cite[Definition~9.21]{marquis:long}). Statement (4) follows from the first assertion of \cite[Lemma~9.26]{marquis:long} (\cite[Lemma~9.26]{marquis:long} is actually stated for $w$ cyclically reduced, but this assumption is not used for the first assertion of that lemma). Finally, (5) is \cite[Proposition~9.29(1)]{marquis:long}.
\end{proof}

\subsection{Length functions}

\begin{Def}\label{Def: length function}
Let $G$ be a group and $Y \subset G$ a symmetric subset. The corresponding \emph{length function} $\ell_Y$ is defined as 
\[
\ell_Y: G\to\mathbb{N} \cup \{ \infty \} \, ; \qquad g\mapsto \min \{n\in \NN\mid g\in Y^n\}
\]
with $Y^n= \{ y_1\cdots y_n\in G\mid y_i\in Y\}$. The identity element $\id$ has length zero. Elements in $G \setminus  \langle Y\rangle $ have length $\infty$.
\end{Def}


In this paper, we will mostly be concerned with a \emph{stabilisation} of the previous notion \cite{wlength}.

\begin{Def}\label{Def: stable length function}
    Let $Y \subset G$ be a symmetric subset, and let $\ell_Y$ be the corresponding length function. The \emph{stable length function} $\sell_Y$ is defined as
    \[\sell_Y(g) \coloneqq \lim\limits_{n \to \infty} \frac{\ell_Y(g^n)}{n},\]
    when $g \in \langle Y \rangle$. If there exists $k \geq 1$ such that $g^k \in \langle Y \rangle$, we set $\sell_Y(g) \coloneqq \frac{\sell_Y(g^k)}{k}$. Otherwise, we set $\sell_Y(g) \coloneqq \infty$. 
\end{Def}

The limit in the definition above exists by Fakete's Lemma. Moreover, for $g \in \langle Y \rangle$ and $k \geq 1$, we have $\sell_Y(g^k) = k \cdot \sell_Y(g)$, so the extension of the domain of $\sell_Y$ is well-defined. We record two general facts.

\begin{Lem}[{\cite[Lemma~2.2]{Avery2023}}]
\label{Lem: length and stable length under homomorphism}
    Let $G, H$ be two groups with conjugacy-invariant symmetric subsets $Y\subseteq G$ and $Z\subseteq H$. Suppose $\varphi \colon G\to H$ is a group homomorphism with $\varphi(Y)\subseteq Z$. Then 
    \[
        \ell_Z(\varphi(g))\leq \ell_Y(g)\; \;\text{and}\;\; \sell_Z(\varphi(g))\leq \sell_Y(g)
    \]
    for all $g\in \langle Y \rangle$.
\end{Lem}

\begin{Lem}
\label{lem length bounded powers}
    Let $g \in \langle Y \rangle$, and suppose that there exists $N \geq 1$ such that $\{ \ell_Y(g^{Nk}) \}_{k \geq 1}$ is bounded. Then $\{ \ell_Y(g^n) \}_{n \geq 1}$ is bounded.
\end{Lem}

\begin{proof}
    Writing $n = kN + r$ for $r < N$ we get
    \[\ell_Y(g^n) \leq \ell_Y(g^{kN}) + \ell_Y(g^r) \leq \sup\limits_{k \geq 1} \ell_Y(g^{kN}) + \max\limits_{0 \leq i < N} \ell_Y(g^i).\]
    The first term is bounded by assumption, and the second term is bounded being a maximum over a finite set of finite values.
\end{proof}

The next two definitions are important examples of (stable) length functions.

\begin{Def}
    \label{Def: stable commutator length}
    Let $G$ be a group and let $C\subseteq G$ be the set of commutators in $G$. The corresponding length function is called \emph{commutator length} and denoted $\cl$; its stabilisation is called \emph{stable commutator length} and denoted $\scl$.
\end{Def}

Computing $\cl$ over free groups is an NP-complete problem \cite{heuer:np}. On the other hand, there is an algorithm for computing $\scl$ over free groups \cite{rationality}, which is even implemented in practice \cite{scallop}. In general, the theory of $\scl$ is much richer than that of $\cl$: we refer the reader to Calegari's book \cite{Calegari2009}, or the surveys \cite{calegari:survey, heuer:survey}.

\begin{Def}
    \label{Def: stable torsion length}
    Let $G$ be a group and let $T$ be the set of all torsion elements in $G$. The corresponding length function is called \emph{torsion length} and denoted $\tl$; its stabilisation is called \emph{stable torsion length} and denoted $\stl$.
\end{Def}

This latter notion was mainly studied by Avery and Chen \cite{Avery2023}, who proved several results parallel to the most celebrated ones on $\scl$. For instance, there is an algorithm for computing $\stl$ over free products of finite groups.

Now, we move to the most important length function in this paper, which is defined specifically for Coxeter groups.

\begin{Def}
    \label{Def: stable reflection length}
    Let $(W,S)$ be a Coxeter system. The conjugates of the standard generators in $S$ are called \emph{reflections}. The set of reflections $R$ generates $W$. The corresponding length function is called \emph{reflection length} and denoted $\rl$; its stabilisation is called \emph{stable reflection length} and denoted $\srl$.
\end{Def}

\begin{Rem}
    Although this definition really only makes sense for Coxeter groups, the reflection length coincides with the \emph{cancellation length} with respect to the finite normal generating set $S$ \cite{Dyer2001}. For the general framework of cancellation length on groups, and its asymptotic properties, we refer the reader to \cite{cancellation}.
\end{Rem}

It is easy to see that all of these functions are additive under direct products, this is \cite[Proposition~1.2]{Mccam2011} for $\rl$, is established similarly for $\tl$ and $\cl$, and implies the same for the stable versions.

\begin{Lem}
\label{lem products}
    Let $W_1, W_2$ be Coxeter groups. Then
    \begin{align*}
        \rl_{W_1 \times W_2}(w_1, w_2) &= \rl_{W_1}(w_1) + \rl_{W_2}(w_2), \\
        \srl_{W_1 \times W_2}(w_1, w_2) &= \srl_{W_1}(w_1) + \srl_{W_2}(w_2).
    \end{align*}
\end{Lem}

Moreover, for Coxeter groups of finite and affine type, formulas for $\rl$ are known (see \cite{Carter1972} and \cite{Lewis2018}). In particular, $\rl$ is bounded on these groups, and therefore $\srl$ vanishes. We deduce:

\begin{Lem}
\label{lem reduction}
    Let $W$ be a Coxeter group, which we decompose as $W_0 \times W_1 \times \cdots \times W_r$, where $W_0$ is the product of its finite and affine components, and $W_1, \ldots, W_r$ are its components of indefinite type. Let $w \in W$, written as $w = w_0 w_1 \cdots w_r$ accordingly. Then $\srl_W(w) > 0$ if and only if $\srl_{W_i}(w_i) > 0$ for some $i \in \{1, \ldots, r\}$, and $\{\rl_W(w^n)\}_{n \in \geq 1}$ is unbounded if and only if $\{\rl_{W_i}(w_i^n)\}_{n \geq 1}$ is unbounded for some $i \in \{1, \ldots, r\}$.
\end{Lem}

An additional useful fact about $\rl$ is that its restriction to a parabolic subgroup coincides with the reflection length of that subgroup \cite[Corollary 1.4]{Dyer2001}. This implies the same fact about $\srl$.

\begin{Lem}
\label{lem srl parabolic}
    Let $W' < W$ be a parabolic subgroup. Then for all $w \in W'$, $\rl_W(w) = \rl_{W'}(w)$ and hence $\srl_W(w) = \srl_{W'}(w)$.
\end{Lem}

\begin{Rem}
    This is a very useful property that will play an important role in the proof of Theorem \ref{intro thm main}. The situation for $\scl$ is different (except in the special case that $W'$ is a retract), in fact it is an open question whether $\scl_{W'}(w) > 0$ implies $\scl_W(w) > 0$ \cite[Remark 1.10]{cancellation}.
\end{Rem}

\subsection{(Stable) reflection length vs (stable) torsion length}

Reflections are torsion elements. Hence Lemma \ref{Lem: length and stable length under homomorphism} implies that $\tl(w) \leq \rl(w)$ and $\stl(w) \leq \srl(w)$. Combined with a known relationship between $\scl$ and $\stl$ \cite[Proposition~1]{Kotschick2004} we obtain:

\begin{Lem}
\label{lem scl vs stl vs srl}
    Let $W$ be a Coxeter group. Then for all $w \in W$:
    \[
        2\scl(w)\leq \stl(w) \leq \srl(w) < \infty.
    \]
\end{Lem}

In fact, the inequality between $\stl$ and $\srl$ is a bi-Lipschitz equivalence.

\begin{Lem}
\label{lem stl srl bilip}
    Let $W$ be a Coxeter group. Then there exists a constant $C = C(W)$ such that for all $w \in W$
    \[\tl(w) \leq \rl(w) \leq C \tl(w),\]
    and similarly for $\stl$ and $\srl$.
\end{Lem}

\begin{proof}
    Clearly the statement for $\tl$ and $\rl$ implies the one for the stable versions. By \cite[Proposition 2.87]{Abramenko2008}, every torsion element of $W$ is contained in a finite parabolic subgroup. It follows that there are finitely many conjugacy classes of torsion elements in $W$, let us choose representatives $t_1, \ldots, t_n$. Letting $C \coloneqq \max_i \rl(t_i)$ we obtain the result.
\end{proof}

Since we are interested in the asymptotic behaviour, from now on we will only focus on $\scl$ and $\srl$. Moreover, we now know that positivity of $\scl$ implies positivity of $\srl$.

\subsection{Two questions}
\label{ss questions}

By analogy with common themes in $\scl$, we propose two motivating questions for future research.

\begin{Qu}
\label{q spectral gap}
    Let $W$ be a Coxeter group. Is there a \emph{spectral gap} in $\srl$ over $W$? Namely, does there exist a constant $C = C(W) > 0$ such that for every $w \in W$ either $\srl(w) > C$ or $\srl(w) = 0$?
\end{Qu}

When $W$ is a \emph{right angled} Coxeter group (Definition \ref{def racg}), a positive answer follows from Lemma \ref{lem scl vs stl vs srl}, Lemma \ref{lem involutions srl} below, and the spectral gap for $\scl$ \cite[Corollary 6.18]{Chen2023}. It is unknown whether a spectral gap in $\scl$ holds for \emph{all} Coxeter groups, but Question \ref{q spectral gap} could be more approachable.

\begin{Qu}
\label{q rationality}
    Let $W$ be a Coxeter group. Is $\srl(w)$ rational, for all $w \in W$?
\end{Qu}

Rationality is a very powerful property for $\scl$ and $\stl$, but it remains an open problem in Coxeter groups. If $W$ is a \emph{universal} Coxeter group, i.e. a free product of cyclic groups of order $2$, then $\scl$ \cite[Theorem A]{chen:products} and $\stl$ \cite[Theorem B]{Avery2023} are rational. In this case, all torsion elements are reflections, so $\srl = \stl$ is rational as well.

\section{Positivity of stable commutator length}
\label{sec. scl}

Thanks to Lemma \ref{lem reduction}, to understand (stable) reflection length, we may reduce to the case in which $W$ is of irreducible indefinite type. Moreover, by Lemma \ref{lem srl parabolic}, when studying the (stable) reflection length of an element $w \in W$, we may assume that $\Pc(w) = W$. In this section, we give a sufficient condition for $\scl_W(w) > 0$, which by Lemma \ref{lem scl vs stl vs srl} implies $\srl_W(w) > 0$.

\begin{Def}
\label{def:chiral}
    An element $g \in G$ is called \emph{achiral} if there exists $m \geq 1$ such that $g^m$ is conjugate to $g^{-m}$. Otherwise, $g$ is called \emph{chiral}.
\end{Def}

This is an algebraic terminology that is commonly used in the literature on $\scl$ (see e.g. \cite{scl:mcg}). For rank one elements in groups acting on CAT(0) spaces, it coincides with the more geometric notion of \emph{irreversible} from \cite{Caprace2009}: see \cite[Lemma 2.2(ii)]{Caprace2009}.

Achirality is an obvious obstruction to the positivity of $\scl$.

\begin{Lem}
\label{Lem: achiral bounded scl}
    If $g \in G$ and $m \geq 1$ are such that $g^m$ is conjugate to $g^{-m}$, then $g^{2km}$ is a commutator for all $k \geq 1$; in particular $\scl(g) = 0$.
\end{Lem}

\begin{proof}
    Let $f \in G$ be such that $fg^{-m}f^{-1} = g^m$. Then $g^{2km} = g^{km} f g^{-km} f^{-1}$.
\end{proof}

The other direction is more interesting.

\begin{Thm}
\label{thm scl positivity}
    Let $W$ be a Coxeter group of irreducible indefinite type. Let $w \in W$ be such that $\Pc(w) = W$. Then $\scl(w) > 0$ if and only if $w$ is chiral.
\end{Thm}

This was essentially achieved by Bestvina--Fujiwara \cite[Main Theorem]{BFCAT0} and Caprace--Fujiwara \cite[Theorem 1.8]{Caprace2009}. However their statements do not immediately give Theorem \ref{thm scl positivity}. Since this is essentially a known result, we only give a minimal proof citing the literature, and refer the reader to those papers for the relevant definitions.

\begin{proof}
    One direction is given by Lemma \ref{Lem: achiral bounded scl}. Consider the proper action of $W$ on the Davis complex $\Sigma(W)$ \cite{Davis2012}, which equipped with an appropriate piecewise Euclidean metric is a proper $\CAT(0)$ space \cite[Theorem~A]{Moussong1988}. The hypothesis implies that $w$ is a rank one element \cite[Proposition~4.5]{Caprace2009}. In particular \cite[Theorem 1.5]{Sisto2016} implies that $W$ is acylindrically hyperbolic (cf. \cite[Theorem 4.4]{Soergel2024}) and $w \in W$ is a generalised loxodromic element. By \cite[Theorem 1.4]{Osin2015}, there is a non-elementary acylindrical (therefore WPD) action of $W$ on a hyperbolic graph such that $w$ is loxodromic. Since moreover $w$ is assumed to be chiral, \cite[Theorem 4.2]{Fournier-Facio2023} implies that there exists a homogeneous quasimorphism $\varphi \colon W \to \RR$ such that $\varphi(w) > 0$. By Bavard duality \cite{bavard}, this implies that $\scl(w) > 0$.
\end{proof}

\section{Chirality and products of involutions}
\label{sec. main}

In this section we characterise chirality in terms of products of involutions (elements which can be expressed as a product of at most two involutions are also known as \emph{strongly real} elements). Let us first observe how this has strong consequences for $\rl$.

\begin{Lem}
\label{lem involutions srl}
    Let $W$ be a Coxeter group, let $w \in W$ and suppose that there exist $a,b \in W$ such that $a^2 = b^2 = \id$ and $w = ab$. Then $\{ \rl(w^n) \}_{n \geq 1}$ is bounded, in particular $\srl(w) = 0$.
\end{Lem}

\begin{proof}
    More precisely, we will show that in this case
    \begin{equation}
    \label{eq:bound}
    \rl(w^n) \leq 
    \begin{cases}
        \rl(a) + \rl(b) \text{ if } n \text{ is odd}; \\
        2 \min\{\rl(a), \rl(b)\} \text{ if } n \text{ is even}.
    \end{cases}
    \end{equation}
    Suppose first that $n = 2j+1$ is odd. Then we write
    \[
    w^n = (a b)^{2j+1} = w^j a w^{-j} b;
    \]
    so $w^n$ is a product of a conjugate of $a$ and $b$. Suppose now that $n$ is even. Then we write
    \[
    w^n = (ab \cdots ba) b (ab \cdots ba)^{-1} b;
    \]
    so $w^n$ is a product of two conjugates of $b$. Similarly, $w^n$ is a product of two conjugates of $a$.
\end{proof}

Here is an equivalent property, which will arise more naturally in our arguments.

\begin{Lem}
\label{lem involutions product conjugate}
    Let $G$ be a group, and let $g \in G$. Then the following are equivalent.
        \begin{enumerate}
        \item $g$ and $g^{-1}$ are conjugate by some $x \in G$ such that $x^2 = \id$.
        \item There exist $a, b \in G$ such that $a^2 = b^2 = \id$ and $g = ab$.
    \end{enumerate}
\end{Lem}

\begin{proof}
    If $xgx=g^{-1}$ for some $x\in G$ with $x^2=\id$, then $(gx)^2=x^2=\id$ and (2) holds with $a=gx$ and $b=x$. Conversely, if $g=ab$ for some $a, b\in G$ with $a^2=b^2=\id$, then $a ga=g^{-1}$.
\end{proof}

The next theorem is the key result, which interprets achirality in terms of products of involutions. We refer the reader to Subsection \ref{ss straight} for the relevant definitions.

\begin{Thm}
\label{thm achiral straight conjugate involutions}
    Let $W$ be a Coxeter group of irreducible indefinite type, and let $w \in W$ with $\Pc(w) = W$. If $w$ is achiral, then $w_c$ and $w_c^{-1}$ are conjugate by an involution.
\end{Thm}

\begin{proof}
    Up to conjugating $w$, we may assume by Lemma~\ref{lem straight part}(4) that $P_w^{\max}=W_I$ for some spherical subset $I\subseteq S$. Let $m \geq 1$ and $x\in W$ such that $x^{-1} w^mx=w^{-m}$. Up to replacing $m$ by some multiple, we may further assume by Lemma~\ref{lem straight part}(2) that $x^{-1} w_c^mx=w_c^{-m}$. 
    Lemma~\ref{lem straight part}(3) then yields
    \[w_c^{-1}=\core(w_c^{-1})=\core(w_c^{-m})=\core(x^{-1} w_c^mx)=\core(x^{-1} w_c x).\]
    Note that
    \[x^{-1} P^{\max}_wx=x^{-1} P^{\max}_{w^m}x=P^{\max}_{x^{-1} w^mx}=P^{\max}_{w^{-m}}=P^{\max}_w\]
    by Lemma~\ref{lem straight part}(1), and hence $x$ normalises $P_w^{\max}=W_I$. Write $x=x_I\overline{x}$ with $x_I\in W_I$ and $\overline{x}$ of minimal length in $xW_I=W_Ix$. Then
    \[x^{-1} w_c x=\overline{x}^{-1}\cdot x_I^{-1} w_c x_I\cdot \overline{x}=\overline{x}^{-1}\cdot (x_I^{-1} w_c x_Iw_c^{-1})w_c\cdot \overline{x},\] with $x_I^{-1} w_c x_Iw_c^{-1}\in W_I$. In particular, $x_I^{-1} w_c x_I \in W_I w_c$ so that $\core(x_I^{-1} w_c x_I)=w_c$ by uniqueness of the core splitting. Lemma~\ref{lem straight part}(4) then yields
    \[\core(x^{-1} w_c x)=\overline{x}^{-1}\core(x_I^{-1} w_c x_I)\overline{x}=\overline{x}^{-1} w_c\overline{x},\]
    and hence $\overline{x}^{-1} w_c\overline{x}=w_c^{-1}$.
    
    In particular, $\overline{x}^2$ commutes with $w_c$. By Lemma~\ref{lem straight part}(5), this implies that $\overline{x}^2=aw_c^n$ for some $a\in W_I$ and $n\in\ZZ$. If $n\neq 0$, then $\Pc(\overline{x}^2)=W$ and $P_{\overline{x}^2}^{\max}=W_I$ by Lemma~\ref{lem straight part}(1), and hence $\core(\overline{x})=\core(\overline{x}^2)=w_c^{\eps}$ by Lemma~\ref{lem straight part}(3), where $\eps\in\{\pm 1\}$ is the sign of $n$. Thus, in that case, $\overline{x}$ has core splitting $\overline{x}=bw_c^{\eps r}$ for some $b\in W_I$ and $r \geq 1$, and hence $w_c^{-1}=\overline{x}^{-1} w_c\overline{x}=b'w_c$ for some $b'\in W_I$. Comparing cores yields $w_c^{-1} = w_c$, contradicting the fact that $w_c$ has infinite order. Therefore, $n=0$ and $\overline{x}^2\in W_I$. Since $\overline{x}$ is the unique element of minimal length in $\overline{x}W_I$, it follows from $\overline{x}W_I=\overline{x}^{-1} W_I$ that $\overline{x}^2=\id$, and we conclude.
\end{proof}

\begin{Rem}
    In \cite[Lemma 4.8]{Caprace2009}, the authors show that, if $w$ is achiral, then $w^k$ is a product of two involutions, where $k$ is the index in $W$ of a torsion free finite index normal subgroup $W_0$. Theorem \ref{thm achiral straight conjugate involutions} removes the passage to a power and recovers that result: if $w = a w_\infty$ for some $a \in P^{\max}_w$, then $w^k = a' w_\infty^k$ for some $a' \in P^{\max}_w$. As $w_\infty^k$ and $w^k$ both belong to $W_0$, the torsion-freeness of $W_0$ implies that $a' = \id$ and hence that $w^k = w_\infty^k$.
\end{Rem}

\begin{Cor}
\label{cor main rank one}
    Let $W$ be a Coxeter group of irreducible indefinite type, and let $w \in W$ with $\Pc(w) = W$. Then the following are equivalent.
    \begin{enumerate}
        \item $\srl(w) = 0$;
        \item $\{ \rl(w^n) \}_{n \geq 1}$ is bounded;
        \item $\scl(w) = 0$;
        \item $w^{mk}$ is a commutator, for some $m \geq 1$ and all $k \geq 1$;
        \item $w$ is achiral;
        \item The core of $w$ is a product of two involutions;
        \item The straight part of $w$ is a product of two involutions.
    \end{enumerate}
\end{Cor}

Note that we cannot formally state that $\{\cl(w^n)\}_{n \geq 1}$ is bounded, since if $w$ is not in the commutator subgroup, then by definition this sequence will take the value $\infty$ infinitely many times.

\begin{proof}
    If $\{ \rl(w^n) \}_{n \geq 1}$ is bounded, then $\srl(w) = 0$. Then by  Lemma \ref{lem scl vs stl vs srl} also $\scl(w) = 0$, which by Theorem \ref{thm scl positivity} implies that $w$ is achiral. Theorem \ref{thm achiral straight conjugate involutions} and Lemma \ref{lem involutions product conjugate} in turn imply that $w_c$ is a product of two involutions. By Lemma \ref{lem involutions product conjugate}, this property passes to powers, and thus the straight part is also a product of two involutions. Passing to a further power, by Lemma \ref{lem straight part}(2), there is some $N \geq 1$ such that $w^N$ is a product of two involutions, and so $\{ \rl(w^{Nk}) \}_{k \geq 1}$ is bounded by Lemma \ref{lem involutions srl}. Then Lemma \ref{lem length bounded powers} implies that $\{ \rl(w^n) \}_{n \geq 1}$ is also bounded.
    This gives the equivalence of all items, except for (4), but (5) $\Rightarrow$ (4) $\Rightarrow$ (3) by Lemma \ref{Lem: achiral bounded scl}.
\end{proof}

\begin{proof}[Proof of Theorem \ref{intro thm main}]
    By Lemma \ref{lem srl parabolic}, we may assume that $\Pc(w) = W$, and by Lemma \ref{lem reduction} we may assume that $W$ is of irreducible indefinite type. Then the result follows from Corollary \ref{cor main rank one}.
\end{proof}

\section{Coxeter elements}
\label{sec. combinatorial}

Recall that a \emph{Coxeter word} is a word in the alphabet $S$ where every generator appears exactly once, and a \emph{Coxeter element} is one represented by a Coxeter word. If $w \in W$ is a Coxeter element, then $\Pc(w) = W$ \cite[Corollary 4.3]{Caprace2009}.

\begin{Thm}
\label{thm cliques}
    Let $(W, S)$ be a Coxeter system of irreducible indefinite type with Coxeter graph $\Gamma$.
    \begin{enumerate}
        \item There exists a Coxeter element that is conjugate to its inverse if and only if $\Gamma$ is bipartite.
        \item Every Coxeter element is conjugate to its inverse if and only if $\Gamma$ is a tree.
    \end{enumerate}
\end{Thm}

Since Coxeter elements are straight \cite{speyer}, and hence coincide with their straight part, Theorem \ref{intro thm combinatorial} is a combination of Theorem \ref{thm cliques} and Corollary \ref{cor main rank one}(2)$\Leftrightarrow$(7). Recall that a graph is bipartite if and only if it has no odd cycle.

\begin{Rem}
    Let us stress that we use the Coxeter convention for $\Gamma$, where two generators are connected by an edge if they \emph{do not} commute. This is the precise opposite of the convention used for right angled Coxeter groups in geometric group theory.
\end{Rem}

We start with a reduction to the right angled case.

\begin{Def}
\label{def racg}
    A Coxeter group is \emph{right angled} if all edges of the Coxeter graph are labeled by $\infty$, that is $m_{ij} \in \{ 2, \infty \}$ for all $i \neq j$. Given a Coxeter system $(W, S)$, its \emph{right angled cover} $(W_r, S)$ is obtained by replacing all labels other than $2$ with infinity. In other words, the group $W_r$ is defined by the sub-presentation of $W$ where we only retain the commuting relations. It comes with a canonical quotient map $W_r \to W$ that restricts to the identity on $S$. Note that the Coxeter graphs of $W$ and $W_r$ differ only by their labels.
\end{Def}

\begin{Lem}
\label{lem:elements:cover}
    Let $(W, S)$ be a Coxeter system and let $(W_r, S)$ be its right angled cover. Then the quotient $W_r \to W$ induces a bijection between Coxeter elements of $W_r$ and Coxeter elements of $W$. If this bijection maps $w_r$ to $w$, then $w_r$ is conjugate to $w_r^{-1}$ in $W_r$ if and only if $w$ is conjugate to $w^{-1}$ in $W$.
\end{Lem}

\begin{proof}
    The surjection $W_r \to W$ restricts to a surjection from the Coxeter elements of $W_r$ to those of $W$. To see injectivity: if two Coxeter words represent the same element in $W$, then by the solution to the word problem \cite{Tits1969} this can be witnessed using only braid moves. Because every generator appears exactly once, the only braid moves that can be applied are commuting relations, which are already available in $W_r$. Finally, from the description of conjugacy classes of Coxeter elements \cite{cc:coxeter}, we see that whether a Coxeter element is conjugate to its inverse can be witnessed by only using the commuting relations, and therefore holds for $w_r \in W_r$ if and only if it holds for its image $w \in W$.
\end{proof}

The proof of Theorem \ref{thm cliques} will reduce to the following special cases.

\begin{Lem}
\label{lem cycle}
    Let $(W, S)$ be a Coxeter system, and let $\Gamma$ be its Coxeter graph.
    \begin{enumerate}
        \item Suppose that $\Gamma$ is an odd cycle. Then no Coxeter element is conjugate to its inverse.
        \item Suppose that $\Gamma$ is a cycle. Then there exists a Coxeter element that is not conjugate to its inverse.
    \end{enumerate}
\end{Lem}

\begin{proof}    
    Suppose that the Coxeter graph $\Gamma$ is a cycle. Given a Coxeter word $\mathtt{w} = s_1 \cdots s_n$, we define its \emph{curl} to be the number of edges $s_i s_j$ such that $i < j$, minus the number of edges $s_i s_j$ such that $i > j$. Combining \cite[Theorem 1.6]{cc:coxeter:old} and \cite[Theorem 1.1]{cc:coxeter}, we see that two Coxeter words represent conjugate elements if and only if they have the same curl. Moreover, $\mathtt{w}^{-1} = s_n \cdots s_1$ has the opposite curl as $\mathtt{w}$. So if $\mathtt{w}$ represents a Coxeter element that is conjugate to its inverse, then it must have curl $0$.

    (1) If the cycle is odd, then the curl of any Coxeter word is odd, in particular non-zero, so no Coxeter element is conjugate to its inverse.

    (2) Let $\mathtt{w} = s_1 \cdots s_n$ be a Coxeter word oriented along the cycle. It has curl $n-1 > 0$, so the Coxeter element it represents is not conjugate to its inverse.
\end{proof}

\begin{proof}[Proof of Theorem \ref{thm cliques}]
    (1) Suppose that $\Gamma$ is bipartite. Choose a bipartition with parts $\{ s_1, \ldots, s_i \}$ and $\{ s_{i+1}, \ldots, s_n \}$. Then $w = (s_1 \cdots s_i) (s_{i+1} \cdots s_n)$ is a product of two involutions, and therefore conjugate to its inverse.
    Conversely, suppose that there exists a Coxeter element $w \in W$ that is conjugate to its inverse. By Lemma \ref{lem:elements:cover}, we may assume that $W$ is right angled. Suppose by contradiction that $\Gamma$ is not bipartite. Pick a minimal odd cycle $\Delta$ in $\Gamma$, and let $V$ be the corresponding parabolic subgroup. By minimality, $\Delta$ is the Coxeter graph of $V$ (i.e. there are no chords). Moreover, the retraction $W \to V$ maps $w$ to a Coxeter element in $V$ that is conjugate to its inverse. This contradicts Lemma \ref{lem cycle}(1).

    (2) We again reduce to the right angled case by Lemma \ref{lem:elements:cover}. If $\Gamma$ is a tree, then all Coxeter elements are conjugate \cite[Proposition 2.3]{cc:coxeter}. Conversely, suppose that $\Gamma$ is not a tree. Let $\Delta$ be a minimal cycle in $\Gamma$, and let $V$ be the corresponding parabolic subgroup. Again, $\Delta$ is the Coxeter graph of $V$, so considering the retraction $W \to V$ we conclude by Lemma \ref{lem cycle}(2).
\end{proof}

\begin{Cor}
\label{cor cliques general}
    Let $(W, S)$ be a Coxeter system.
    \begin{enumerate}
        \item There exists a Coxeter element $w$ such that $\{\rl(w^n)\}_{n \geq 1}$ is bounded if and only the Coxeter graph of every indefinite component is bipartite.
        \item $\{\rl(w^n)\}_{n \geq 1}$ is bounded for every Coxeter element $w$ if and only if the Coxeter graph of every indefinite component is a tree.
    \end{enumerate}
\end{Cor}

\begin{proof}  
    A Coxeter element in $W$ is a product of Coxeter elements of each component. Combine Lemma \ref{lem reduction}, Corollary \ref{cor main rank one} and Theorem \ref{thm cliques}.
\end{proof}

\bibliographystyle{amsplain}
\bibliography{ref.bib}
\end{document}